\documentclass[12pt,a4paper]{amsart}

\usepackage{amsfonts, amsmath, amssymb, amsthm, amscd, hyperref}

\newtheorem{thm}{Theorem}
\newtheorem*{thm*}{Theorem}
\newtheorem{lem}{Lemma}

\newtheorem{cor}[thm]{Corollary}

\theoremstyle{definition}
\newtheorem{defn}{Definition}

\theoremstyle{remark}

\newcommand{\Sy}{\Sigma^{(2)}}
\newcommand{\Sg}{\Sigma}

\DeclareMathOperator{\cat}{cat}

\DeclareMathOperator{\dist}{dist}

\renewcommand{\int}{\mathop{\rm int}}

\renewcommand{\epsilon}{\varepsilon}

\begin{document}

\title[Multiplicity of continuous maps\dots]{Multiplicity of continuous maps between manifolds}

\author{R.N.~Karasev}
\thanks{This research is supported by the Dynasty Foundation, the President's of Russian Federation grant MK-113.2010.1, the Russian Foundation for Basic Research grants 10-01-00096 and 10-01-00139, the Federal Program ``Scientific and scientific-pedagogical staff of innovative Russia'' 2009--2013}

\email{r\_n\_karasev@mail.ru}
\address{
Roman Karasev, Dept. of Mathematics, Moscow Institute of Physics
and Technology, Institutskiy per. 9, Dolgoprudny, Russia 141700}



\begin{abstract}
We consider a continuous map $f :M\to N$ between two manifolds and try to estimate its multiplicity from below, i.e. find a $q$-tuple of pairwise distinct points $x_1,\ldots, x_q\in M$ such that $f(x_1) = f(x_2) = \ldots = f(x_q)$.

We show that there are certain characteristic classes of vector bundle $f^*TN-TM$ that guarantee a bound on the multiplicity of $f$. In particular, we prove some non-trivial bound on the multiplicity for a continuous map of a real projective space of certain dimension into a Euclidean space.
\end{abstract}

\maketitle

\section{Introduction}

In this paper we consider a continuous map $f :M\to N$ between two manifolds and try to find some sufficient conditions for existence of \emph{multiple points}, i.e. the $q$-tuples of pairwise distinct point $x_1,\ldots, x_q\in M$ such that 
$$
f(x_1) = f(x_2) = \ldots = f(x_q).
$$

We call such a $q$-tuple $x_1,\ldots, x_q$ a \emph{coincident $q$-tuple}, and call the \emph{multiplicity} of $f$ the maximum $q$ such that there exists a coincident $q$-tuple for $f$.

The results of this kind for double points of continuous maps were obtained in~\cite{wu1958,cf1960,schw1966,mcc1978}. They have obvious relation to embeddability and immersibility of manifolds. Some results about the multiplicity in the case $f$ is a smooth immersion are also known, see~\cite{herb1981,eccgr2006} for example. In the recent paper~\cite{grom2010} the lower bounds for the multiplicity are given in the case when the domain space is a polyhedron of high enough complexity (a skeleton of a simplex) and $f$ is a generic piecewise linear or piecewise smooth map. In this paper we investigate the case when $f$ is continuous without any other restrictions.

In Section~\ref{coind-char} we show that there exist certain characteristic classes of vector bundle $f^*TN-TM$ that guarantee the existence of multiple points for $f$. Then we give some particular applications of these classes. We prove Theorem~\ref{proj-coinc} on the multiplicity for continuous maps from a projective space to a Euclidean space, calculate some characteristic classes of coincident $4$-tuples in Section~\ref{4-fold-coinc}, calculate the characteristic classes modulo a prime $p$ of coincident $p$-tuples in Section~\ref{p-fold}.

In Section~\ref{genus-cat} we consider another question, having a lot in common multiple points of maps, that is the question of estimating from below the Krasnosel'skii-Schwarz genus and the Lyusternik-Schnirelmann category of configuration spaces of manifolds. 

The author thanks S.A.~Melikhov for pointing out the relation of this problem to the singularity theory and detailed discussions of the subject.

\section{Local multiplicity of generic smooth maps -- the approach of singularity theory}
\label{generic-smooth}

The theory of singularities for smooth maps gives some approach to multiplicity. For example, it is known that a generic (in some sense) smooth map $f:M\to N$ may have singularities of type $\Sigma^{1^k}$, with the following canonical form~\cite{morin1965}.

Let the local coordinates be $(x_1,\ldots, x_m)$ in $M$ and $(y_1,\ldots, y_n)$ in $N$. Then the map is given by
\begin{eqnarray}
y_i &=& x_i,\ i=1,\ldots, m-1\\
\label{pol-mid}
y_i &=& \sum_{l=1}^k x_{(i-m)k + l} x_m^l,\ i=m,\ldots, n-1\\
\label{pol-last}
y_n &=& \sum_{l=1}^{k-1} x_{(n-m)k + l} x_m^l + x_m^{k+1},
\end{eqnarray}
here the inequality $k(n-m+1)\le m$ must hold. If we select the numbers $x_i$ ($i=1,\ldots m-1$) so that the polynomials in $x_m$ in the right part of (\ref{pol-mid}) are zero, and the right part of (\ref{pol-last}) has $k+1$ distinct roots, then we obtain a coincident $(k+1)$-tuple, since the coordinate $x_m$ has $k+1$ possible choices.  

Such singularities for generic maps are guaranteed by the appropriate characteristic classes of the virtual bundle $f^*TN-TM$. The classes $\sigma_{k, m-n}$ for singularities of type $\Sigma^{1^k}$ in codimension $m-n$ can be expressed in terms of Stiefel-Whitney classes by some recurrent formulas, see~\cite{port1971} for example. 

Unlike the singularity theory approach, the approach to multiplicity in this paper is valid for arbitrary continuous maps, not only smooth and generic. This approach has some similarities with the singularity theory, in particular, some characteristic classes of $f^*TN- TM$ that guarantee multiple points are introduced. In particular, in Section~\ref{coind-char} we introduce a characteristic class $s_{q,m-n}$ that guarantee a coincident $q$-tuple for $q=2^l$, the author does not known whether the classes $s_{q,m-n}$ are a particular case of the classes $\sigma_{q,m-n}$. 

\section{Global multiplicity and configuration spaces}

In this section we consider a continuous map $f :M \to N$ and try to give sufficient conditions for the existence of a coincident $q$-tuple. The most straightforward approach is to consider the configuration space.

\begin{defn} For a topological space $X$ denote the \emph{configuration space}
$$
K^q(X) = \{(x_1, \ldots, x_q)\in X^q : \forall i,j\ x_i\not= x_j\}.
$$
\end{defn}

Note that the permutation group $\Sg_q$ acts freely on $K^q(X)$. For any continuous map $f: M\to N$ denote its power $f^q : K^q(M)\to N^q$ its power restricted to $K^q(M)$. This is an $\Sg_q$-equivariant map. A coincident $q$-tuple is an intersection of $f^q(K^q(X))$ with the (thin) diagonal $\Delta(N)\subseteq N^q$. 

Thus the preimage of the diagonal $(f^q)^{-1}(\Delta(N))\subseteq K^q(M)$ can be considered as an obstruction to deforming the map $f$ so that it has no coincident $q$-tuples. In the case when $M$ and $N$ are smooth manifolds of dimensions $m$ and $n$ respectively, and $M$ is compact, this preimage of the diagonal can be considered as an $\Sg_q$-equivariant cohomology class 
$$
s_q(f)\in H_{\Sg_q}^{n(q-1)}(K^q(M)) = H^{n(q-1)}(K^q(M)/{\Sg_q}),
$$ 
the coefficients of the cohomology being $Z_2$, or $\mathbb Z$, possibly with the sign action of the group $\Sg_q$ depending on the orientability of $N$ and parity of its dimension. Certainly, this obstruction can also be considered as an oriented or non-oriented cobordism class in $\Omega_{(S)O}^{n(q-1)}(K^q(M)/{\Sg_q})$, but we do not use the cobordism in this paper.

The global cohomology class for double points has certain relation with the local double points in the case $N=\mathbb R^n$, see~\cite{wu1958,schw1966,mcc1978} for example. For multiplicity $>2$ we restrict ourselves to studying the local (in some sense) multiplicity in the following sections.

\section{Local multiplicity of continuous maps -- the configuration space bundle}

In order to prove that the class $s_q(f)$ is nonzero, it sometimes makes sense to restrict it to the intersection of $K^q(M)$ with a certain neighborhood of the thin diagonal $\Delta(M)\subset M^q$. This approach would give multiple points, that are close enough to each other in $M$. We call such multiple points \emph{local}. 

Note that this type of local multiplicity is stronger than the local multiplicity in the smooth generic case (see Section~\ref{generic-smooth}), because here we may guarantee the existence of a coincident $q$-tuple $\{x_1,\ldots, x_q\}$ with some small but bounded from below diameter, the bound depending on $M$ only. In will be clear from the definition of $Q^q(M,\ldots)$ in Section~\ref{coind-char}.

Similar to what is done in the singularity theory, we are going to reformulate the problem as a problem for bundle maps. Suppose $M$ is a compact Riemannian manifold, and $N$ is a Riemannian manifold with the injectivity radius $r(M)$.

Consider the tangent bundle $TM$ and the exponential map $\exp :TM\to M\times M$, induced by the Riemannian metric, and sending a tangent vector $\tau$ at $x$ to the ends of a geodesic $(x, y)$ with length $|\tau|$ and starting direction $\tau$. Let us fix some $x\in M$, then the $q$-th power $\exp^q : K^q(T_xM) \to M^q$ maps $q$-tuples $(\tau_1,\ldots, \tau_q)\in (T_xM)^q$ of vectors with lengths $<r(M)$ to $q$-tuples of distinct points in $M$. Hence by replacing $T_xM$ with an open disc $D_xM$ of radius $r(M)$ we obtain a well-defined map $\exp^q : K^q(D_xM)\to K^q(M)$. Hence there exists $\Sg_q$-equivariant map
$$
\exp^q : K_M^q(DM) \to K^q(M),
$$
here we assume the following definition.

\begin{defn}
Let $\xi : E(\xi)\to M$ be a vector or disc bundle over $M$. The subspace of $K^q(E(\xi))$, consisting of configurations lying in the same fiber of $\xi$, is denoted $K_M^q(\xi)$ and called the \emph{configuration space bundle}.
\end{defn}

Now let us consider a continuous map $f :M\to N$. Let us take small enough tangent disc bundle $DN$ so that $\exp$ (of $N$) is invertible on it. Then take small enough tangent disc bundle $DM$ so that the inclusion for the map of pairs $f^2\circ\exp_M DM\subset \exp_N DN$ holds. Then a fiberwise map 
$$
\phi = \exp_N^{-1}\circ f^2\circ \exp_M : DM\to DN
$$
is defined. To find a local coincident $q$-tuple it is sufficient to find a coincident $q$-tuple of the fiberwise map $\phi^q : K_M^q(DM)\to DN^{\oplus q}$, the latter is the $q$-fold Whitney sum of bundles over $N$. Consider the pullback $f^*(DN)$ with the natural fiberwise map $f_* : f^*(DN)\to DN$. Note, that there exists a natural fiberwise map $\psi : DM\to f^*(DN)$, and the corresponding map $\psi^q : K_M^q(DM) \to (f^*(DN))^{\oplus q}$ such that $\phi^q = f_* \circ \psi^q$. So the local multiplicity is bounded from below if we bound from below a multiplicity in the fiberwise map 
$$
\psi : DM \to f^*(DN)
$$
over the same space $M$.

In the sequel we do not distinguish between a vector bundle and its disc bundle, since they and their configuration spaces are diffeomorphic. Let us generalize a problem to finding multiple points for a fiberwise map $\psi : \xi\to \eta$ of some vector bundles $\xi$ and $\eta$ over the same space $M$.

Let $A_q$ be the $q-1$-dimensional representation of $\Sg_q$, arising from the natural permutation representation on $\mathbb R^q$ by taking the quotient $\mathbb R^q/(1,1,\ldots,1)$. Denote the natural projection of the configuration space bundle $\xi^q : K_M^q(\xi)\to M$. Now the map $\psi^q$ gives an $\Sg_q$-equivariant section of the vector bundle $(\xi^q)^*(\eta^{\oplus q})$. Composed with the natural projection 
$$
(\xi^q)^*(\eta^{\oplus q})\to A_q\otimes(\xi^q)^*(\eta),
$$ 
it also gives a section $\psi_0^q$ of the vector bundle $A_q\otimes(\xi^q)^*(\eta)$ over $K_M^q(\xi)$. Note that $\psi_0^q$ is an equivariant section w.r.t the natural action of $\Sg_q$ on $K_M^q(\xi)$ and on $A_q$. Now the coincident $q$-tuples correspond to the zero set of the section $\psi_0^q$, and we have reduced the problem of finding local multiple points to proving that the zero set of the $\Sg_q$-equivariant vector bundle $A_q\otimes(\xi^q)^*(\eta)$ over $K_M^q(\xi)$ is nonempty.

\section{The Euler class of local multiple points}

Now let us consider the zero set $Z(\psi^q, \xi, \eta)$ of some $\Sg_q$-equivariant section of the bundle $A_q\otimes(\xi^q)^*(\eta)$ over $K_M^q(\xi)$. We shall denote the bundle simply $A_q\otimes \eta$ since it does not lead to a confusion.

Certainly, the set $Z(\psi^q, \xi, \eta)$ may be considered a manifold for generic sections, and it is Poincare dual to the Euler class $e(A_q\otimes \eta)$, taken in the equivariatn cohomology or some bordism theory. It can be calculated directly sometimes, but we are going to ``stabilize'' it in some sense to simplify the computation, though some information can be lost. 

Consider some other vector bundle $\zeta$ and two fiberwise maps $\psi : \xi\to \eta$ and $\iota : \zeta\to\zeta$, the latter being the identity. Now it is readily seen that 

\begin{equation}
\label{coinc-stability}
Z((\psi\oplus\iota)^q, \xi\oplus\zeta, \eta\oplus\zeta) = Z(\psi^q, \xi, \eta)\times_M \zeta.
\end{equation}

Hence if the Euler class $e(A_q\otimes(\eta\oplus\zeta))$ is nonzero over $K_M^q(\xi\oplus\zeta)$, then the set $Z(\psi^q, \xi, \eta)$ cannot be empty.

\begin{lem}
\label{triv-reduction}
The coincident $q$-tuples in a continuous fiberwise map $\psi : \xi\to\eta$ over $M$ are guaranteed by a nonzero $\Sg_q$-equivariant Euler class $e(A_q\otimes(\xi^\perp\oplus\eta))$ in the cohomology (or cobordism) of $K^q(\mathbb R^\mu)\times M$, where $\mu=\dim \xi + \dim \xi^\perp$.
\end{lem}

\begin{proof}
In the above reasoning take $\zeta$ to be $\xi^\perp$ such that $\xi\oplus\xi^\perp = \epsilon^\mu$, $\epsilon$ denotes a trivial vector bundle. Then it suffices to note that 
$$
K_M^q(\epsilon^\mu) = K^q(\mathbb R^\mu)\times M.
$$ 
\end{proof}

If we consider the cohomology with coefficients in a field $Z_p$, then by the K\"unneth formula the algebra $H_{\Sg_q}^*(K^q(\mathbb R^\mu)\times M)$ is a free $H^*(M)$-module, spanned by the linear basis of $H^*(K^q(\mathbb R^\mu)/{\Sg_q})$. Denote the latter basis by $(h_1,\ldots, h_{N(q,\mu)})$, it is known (see~\cite{fuks1970,vass1994,rot2008}) that in the case $p=2$ the elements $h_i$ can be selected to be part of the basis of $H^*(B\Sg_q, Z_2)$.

Now we can decompose the Euler class 
$$
e(A_q\otimes (\xi^\perp\oplus\eta)) = \sum_{i=1}^{N(q,\mu)} s_ih_i
$$
and obtain the elements $s_i\in H^*(M, Z_p)$, that depend naturally on the bundle $\xi^\perp\oplus\eta$. Hence $s_i$ only depend on the Stiefel-Whitney (in case $p=2$) or the Pontryagin and Euler (in case $p$ odd) classes of the virtual bundle $\eta-\xi$, since it is sufficient to consider the situation over some Grassmann variety and then use the naturality of construction. This is in accordance with the similar result for smooth map singularities, where the characteristic classes depend on the virtual bundle $f^*TN-TM$.

\section{Local coincident $p^k$-tuples and their characteristic classes}
\label{coind-char}

Let us consider the case when $q$ is a power of a prime $p$. In this case the cohomology $H^*(K^q(\mathbb R^\mu)/{\Sg_q}, A)$ is zero in dimensions $>(q-1)(\mu-1)$, and its $(q-1)(\mu-1)$-dimensional cohomology is generated by the class $e(A_q)^{n-1}$, see~\cite{vass1988,ossa1996,rot2008,kar2009} for different cases of these results. Here the coefficients $A$ are $Z_p$ for $p=2$ or odd $\mu$, and $Z_p$ with sign action of $\Sg_q$ in other cases.

Thus similar to the above definition, we can put 
$$
e(A_q\otimes (\xi^\perp\oplus\eta)) = s_{q, d}(\xi^\perp\oplus\eta) e(A_q)^{\mu-1} +\ldots,
$$
where $d=\dim\eta-\dim\xi = \dim(\xi^\perp\oplus\eta) - \mu$ and $\ldots$ denotes the terms with the dimension of the corresponding class in $H^*(M, Z_p)$ larger, then that of $s_{q, d}$. Note that the dimension of $s_{q, d}$ equals $(q-1)(d+1)$.

\begin{defn}
We call $s_{q, d}(\eta-\xi)$ the \emph{leading characteristic class} of local coincident $q$-tuples for prime powers $q$ at codimension $d$.
\end{defn}

In some cases this class (and possibly some higher classes) can be calculated.

We are going to prove the following result, showing that $s_{q, d}$ are nontrivial in the case of $q=2^k$.

\begin{thm}
\label{2k-fold-leading}
Let $q$ be a power of two. Denote $w_i$ the Stiefel-Whitney classes of $\eta-\xi$. Then as a polynomial in the Stiefel-Whitney classes of $\eta-\xi$
$$
s_{q, d}(\eta-\xi) \equiv w_{d+1}^{q-1} \mod w_{d+2}, w_{d+3}, \ldots.
$$
\end{thm}

In order to prove Theorem~\ref{2k-fold-leading} we are going to consider some subspace of the configuration space $K^q(\mathbb R^n)$. Such subspaces were introduced in~\cite{hung1990} and proved to be very useful in describing the cohomology of the symmetric group modulo $2$.

\begin{defn}
Let $q=2^k$ and consider a sequence $\delta_1, \ldots, \delta_k$ of positive integers such that for any $l<k$
$$
\delta_l > \sum_{i=l+1}^k \delta_i.
$$
Let $Q_0^1(\mathbb R^n)$ be the configuration, consisting of one point at the origin. 

Let by induction $Q_0^q(\mathbb R^n, \delta_1, \ldots, \delta_k)$ be the set of all $q$-point configurations, such that the first $q/2$ points form a configuration of $Q_0^{q/2}(\mathbb R^n, \delta_2, \ldots, \delta_k)$, shifted by a vector $u$ of length $\delta_1$, and the other $q/2$ points form a configuration of $Q_0^{q/2}(\mathbb R^n, \delta_2, \ldots, \delta_k)$, shifted by a vector $-u$.
\end{defn}

\begin{defn}
\label{Qq-defn2}
A configuration in $Q_0^q(\mathbb R^n, \delta_1, \ldots, \delta_k)$ can also be described inductively as $x_1,\ldots, x_q\in\mathbb R^n$ such that all the distances $\dist(x_{2i-1}, x_{2i}) = 2\delta_k$ and the midpoints of $[x_{2i-1}, x_{2i}]$ form a configuration of $Q_0^{q/2}(\mathbb R^n, \delta_1, \ldots, \delta_{k-1})$. 
\end{defn}

Note that $Q_0^q(\mathbb R^n, \delta_1, \ldots, \delta_k)$ is always a product of $q-1$ spheres of dimension $n-1$, and we shall omit $\delta_i$ in the notation since it does not change the diffeomorphism type of $Q^q(\mathbb R^n)$. Then we can naturally define the space $Q_M^q(\xi)\subset K_M^q(\xi)$ for any vector bundle $\xi : E(\xi)\to M$ as a bundle of corresponding $Q_0^q(\xi^{-1}(x))$ for $x\in M$.

Note that the Definition~\ref{Qq-defn2} (distance and midpoint characterization) can be applied to any Riemannian manifold $M$, if we allow the last center point (configuration $Q_0^1$) be any $x\in M$. The distances should be chosen small enough in order for the midpoint to be unique.

\begin{defn}
Let $M$ be a Riemannian manifold. Define $Q^q(M, \delta_1, \ldots, \delta_k)\subset K^q(M)$ for $q=2^k$ inductively as follows. 

$Q^1(M) = M$.

For $q\ge 2$ let $Q^q(M, \delta_1, \ldots, \delta_k)$ be the set of $q$-tuples $x_1,\ldots, x_q\in M$ such that all the distances $\dist(x_{2i-1}, x_{2i}) = 2\delta_k$ and the midpoints of $[x_{2i-1}, x_{2i}]$ form a configuration of $Q^{q/2}(M, \delta_1, \ldots, \delta_{k-1})$. 
\end{defn}

The following lemma allows to guarantee not only local singularities, but singularities of some finite size from considering $Q_M^q(TM)$ as the configuration space.

\begin{lem}
\label{metric-Qq}
Let the injectivity radius of $M$ be $r$ and for all $i=1,\ldots, k$
$$
2\delta_i < r.
$$
Then $Q^q(M, \delta_1, \ldots, \delta_k)$ is a fiber bundle (the bundle map is the last stage midpoint) over $M$, and is naturally homeomorphic to $Q_M^q(TM)$
\end{lem}

\begin{proof}
Let us prove by induction. For any configuration $(x_1,\ldots, x_q)\in Q^q(M,\delta_1, \ldots, \delta_k)$ the midpoints of pairs $[x_1, x_2], [x_2,x_3],\ldots,[x_{q-1}, x_q]$ form a configuration in $Q^{q/2}(M,\delta_1, \ldots, \delta_{k-1})$. Since $2\delta_k<r$, then knowing the midpoint of $[x_1,x_2]$, the possible positions of the points $x_1,x_2$ form a sphere. 

So $Q^q(M,\ldots)$ is a product-of-spheres bundle over $Q^{q/2}(M,\ldots)$. Moreover, these spheres are spheres of the vector bundles $\pi_i^*(TM)$, where $\pi_i : Q^{q/2}(M,\ldots)\to M$ is the map, assigning to a configuration its $i$-th point. Note that the maps $\pi_i$ are all homotopic to the centerpoint map $\pi :Q^{q/2}(M,\ldots)\to M$ (the homotopy can be obtained by deforming a point $x_{2i-1}$ or $x_{2i}$ to the midpoint of $[x_{2i-1}, x_{2i}]$, and then repeating inductively), hence all the vector bundles are equivalent to $\pi^*(TM)$. Now the proof is completed by applying the inductive assumption.
\end{proof}

The space $Q_0^q$ (or $Q^q$) is not invariant under the natural $\Sg_q$-action, but it is invariant under the action of a certain Sylow subgroup.

\begin{defn}
Let $q=2^k$. Denote $\Sy_q$ the Sylow subgroup of $\Sg_q$, generated by all permutations of two consecutive blocks $[a2^l + 1, a2^l + 2^{l-1}]$ and $[a2^l + 2^{l-1}+1, (a+1)2^l]$, where $2\le l \le k$ and $0\le a \le 2^{k-l}-2$.
\end{defn}

\begin{lem}
\label{Qq}
The manifold $Q_0^q(\mathbb R^n)$ is $\Sy_q$-invariant. The cohomology $H_{\Sy_q}^{(q-1)(n-1)}( Q^q(\mathbb R^n), Z_2)$ is generated by the Euler class $e(A_q)^{n-1}$.
\end{lem}

\begin{proof}
The first claim is obvious by definition.

Consider the natural projection $\psi : \mathbb R^n\to \mathbb R^{n-1}$, defined by
$$
\psi(x_1,\ldots,x_{n-1},x_n) = (x_1,\ldots,x_{n-1}).
$$ 
For this projection the only configurations in $Q^q(\mathbb R^n)$ that give coincident $q$-tuples are those with all coordinates zero except $x_n$. But such configurations form exactly one orbit of $\Sy_q$. Since this orbit is Poincare dual to the Euler class $e(A_q)^{n-1}$, which is responsible for coincident $q$-tuples in this case, we see that $e(A_q)^{n-1}$ coincides with the fundamental class of the manifold $Q^q(\mathbb R^n)/\Sy_q$.
\end{proof}

In is well-known~\cite{bro1982}, that if we consider the $\Sg_q$-equivariant cohomology with coefficients $Z_p$, then the cohomology does not change when passing to $p$-Sylow subgroup. Here we do not only pass to a Sylow subgroup, but also refine the configuration space $K^q$ to a manifold, to allow some direct geometric reasoning as in the proof of Lemma~\ref{Qq}.

Using Lemma~\ref{Qq} the leading characteristic class of coincident $2^k$-tuples can be defined as follows.

\begin{defn}
Denote $\pi : Q_M^q(\xi)/\Sy_q \to M$ the natural projection. Then 
$$
s_{q, d}(\eta-\xi) = \pi_!(e(A_q\otimes \pi^*\eta)),
$$
i.e. geometrically it is a projection of the set of coincident $q$-tuples in $Q_M^q(\xi)/\Sy_q$ to $M$.
\end{defn}

This definition is the same because for the restricted set of coincident $q$-tuples $Z(\psi, \xi, eta)\subseteq Q^q$ of a fiberwise map $\psi : \xi\to \eta$ the stability holds in the following exact form
\begin{equation}
\label{coinc-stability-pow2}
Z((\psi\oplus\iota)^q, \xi\oplus\zeta, \eta\oplus\zeta) = Z(\psi^q, \xi, \eta).
\end{equation}
Then passing to the case of trivial $\xi$ we see that the topmost cohomology of $Q_0^q(\mathbb R^\mu)$ is the same as in $K^q(\mathbb R^\mu)$. The map $\pi_!$ ``divides'' by the fundamental class of $Q_0^q(\mathbb R^\mu)/\Sy_q$ in $H^*(Q_0^q(\mathbb R^\mu)/\Sy_q\times M)$, similar to the first definition of $s_{q, d}$.
 
\section{Proof of Theorem~\ref{2k-fold-leading} and some corollaries}

Now we are ready to prove Theorem~\ref{2k-fold-leading} using the geometric definition of $s_{q,d}$.

Let us find $s_{q,d}$ for a certain fiberwise map over the Grassmannian $M=G_{n, d+1}$ of linear $n$-subspaces in $\mathbb R^{n+d+1}$. Denote the canonical $n$-dimensional bundle $\gamma : E(\gamma)\to G_{n, d+1}$. Now consider the map 
$$
f :\mathbb R^{n+d+1}\to \mathbb R^{n+d}
$$
given by
$$
f(x_1, \ldots, x_{n+d+1}) = (x_1 - x_{n+d+1}^2, x_2 - x_{n+d+1}^3,\ldots, x_{n+d} - x_{n+d+1}^{n+d+1}).
$$
Each fiber of $\gamma$ is mapped with this map to $\mathbb R^{n+d}$, so $f$ can be considered as a fiberwise map of $\gamma$ to $\epsilon^{n+d}$. 

Let us describe the coincident $q$-tuples of $f$ in $Q_M^q(\gamma)$. First, note that there is a natural inclusion $Q_M^q(\gamma)\to Q_0^q(\mathbb R^{n+d+1})$. A configuration of $q$ points $(p_1, \ldots, p_q)\in Q_0^q(\mathbb R^{n+d+1})$ is mapped to one point $y$ if they lie on a single curve $C(c_1, \ldots, c_{n+d})$, given by the parameterization
$$
x_1 = c_1 + t^2,\ x_2 = c_2 + t^3,\ \ldots,\ x_{n+d} = c_{n+d} + t^{n+d+1},\ x_{n+d+1} = t.
$$

Consider a single curve $C(c_1,\ldots,c_{n+m})$ and the set $Z(c_1,\ldots, c_{n+m})$ of all configurations in $Q^q(\mathbb R^n)$ (not $Q_0^q$!), lying entirely on $C(c_1,\ldots,c_{n+d})$. 

Consider a configuration $(p_1,\ldots, p_q)\in Z(c_1,\ldots, c_{n+m})$ and assume that the points are ordered w.r.t. the coordinate $x_{n+d+1}$, which corresponds with the parameter on the curve. From Definition~\ref{Qq-defn2} it is clear that if $\delta_i$ are small enough in the definition (so that the curvature of $C(c_1,\ldots, c_{n+m})$ becomes negligible), then the configuration $(p_1,\ldots, p_q)\in Z(c_1,\ldots, c_{n+m})$ is determined uniquely by any one point $p_i$, which can be chosen arbitrarily. In other words, $p_i$ is a smooth parameter on $Z(c_1, \ldots, c_{n+m})/\Sy_q$.

Denote 
$$
Z=\bigcup_{c_1,\ldots,c_{n+d}\in\mathbb R} Z(c_1, \ldots, c_{n+d})\subset Q^q(\mathbb R^{n+d+1}).
$$
We have already noted that the map $g_i : Z/\Sy_q\to\mathbb R^{n+d+1}$ taking any configuration to its $i$-th point w.r.t. the coordinate $x_{n+d+1}$ is a diffeomorphism. For any configuration $(p_1,\ldots, p_q)\in Z$ put 
$$
h(p_1,\ldots, p_q) = \sum_{i=1}^q g_i(p_1, \ldots, p_q),
$$ 
i.e. the center point of the configuration. The map $h$ is smooth on $Z/\Sy_q$ and for small enough $\delta_i$ it is a diffeomorphism onto $\mathbb R^{n+d+1}$. Hence $h^{-1}(0)$ is the only configuration in $Q_0^q(\mathbb R^{n+d+1})$ that is mapped into single point by $f$.

Since the configuration $h^{-1}(0)$ lies on a translate of the moment curve, which is a convex curve, then its points span some $q-1$-dimensional linear subspace $L\subset\mathbb R^{n+d+1}$. Now any linear space $V\in G_{n, d+1}$, that is supposed to have coincident $q$-tuples in the fiber $Q_0^q(V)$, must contain $L$. Moreover, it can be easily seen that the map $f$ is transversal to zero and the condition $V\supseteq L$ defines the Poincare dual to $s_{q, d}$ homology class. From the well-known description of the (dual) Stiefel-Whitney classes it follows that 
\begin{equation}
\label{dual-sw-power}
s_{q, d} = w_{d+1}^{q-1}(\gamma^\perp)
\end{equation}
in the cohomology $H^*(G_{n, d+1}, Z_2)$. Now it suffices to note that this cohomology algebra has generators $w_1(\gamma), \ldots, w_n(\gamma)$ and relations
\begin{equation}
\label{grass-gen}
w_{d+2}(\gamma^\perp) = w_{d+3}(\gamma^\perp) = \dots = 0,
\end{equation}
hence (\ref{dual-sw-power}) holds over arbitrary space modulo higher Stiefel-Whitney classes of $\eta - \xi$, because the number $n$ can be taken arbitrarily large, and the relations in (\ref{grass-gen}) are the only essential relations.

\begin{cor} 
\label{2-coinc}
For double points we have:
$$
s_{2, d} = w_{d+1}.
$$
\end{cor}

The corollary follows from Theorem~\ref{2k-fold-leading} because in the right part there cannot be anything, depending on $w_{d+2}, w_{d+3},\ldots$ from the dimension considerations.

Corollary~\ref{2-coinc} was actually proved in~\cite{cf1960} for fiberwise maps to trivial bundle. Moreover, it is known that the local double points for maps $M\to\mathbb R^n$ have a relation to global double points for maps $M\to\mathbb R^{n+1}$ (informally, they have the same characteristic class), see~\cite{schw1966} for algebraic description, or~\cite{mcc1978} for some geometric reasoning. 

Let us prove a theorem that gives coincident $q$-tuples for maps of certain projective spaces to $\mathbb R^n$ by the class $s_{q, d}$.

\begin{thm}
\label{proj-coinc}
Suppose that $q$ is a power of two, $q(d+1) < 2^l-1$. Then any continuous map
$$
f : \mathbb RP^{2^l-2-d}\to \mathbb R^{2^l-2}
$$
has multiplicity $\ge q$.
\end{thm}

Note that this theorem gives a coincident $q$-tuple on a configuration from any subspace
$$
Q^q(\mathbb RP^{2^l-2-d}, \delta_1,\ldots, \delta_k)\subset K^q(\mathbb RP^{2^l-2-d})
$$
for any sequence of $\delta_i$, satisfying 
$$
\delta_1 < \pi/4,\forall i\ \delta_i>\delta_{i+1}+\dots+\delta_k.
$$
Here we measure the distance on $\mathbb RP^n$ as the angle between lines.

\begin{proof}
Any map $f : \mathbb RP^m\to \mathbb R^n$ induces a fiberwise map between $\xi=T\mathbb RP^m$ and $\epsilon^n$. Suppose we have some normal bundle $\xi^\perp$ of dimension $k$. The Stiefel-Whitney class of $\epsilon^n-\xi$ is 
$$
w(\epsilon^n-\xi) = (1+u)^{-m-1} = (1+u)^{2^l-m-1} = (1+u)^{d+1},
$$
and by Theorem~\ref{2k-fold-leading} we have $s_{q,d}(\epsilon^n-T\mathbb RP^m) = u^{(q-1)(d+1)}$, which is nonzero since $(q-1)(d+1)\le m=2^l-2-d$.
\end{proof}

\section{Characteristic classes for local coincident $4$-tuples -- calculations}
\label{4-fold-coinc}

Let us give some general schema of calculating $s_{q, d}$ for any particular $q$ and $d$. The class $e(A_q\otimes (\xi^\perp\oplus \eta))$ can be calculated in the assumption that the bundle $\xi^\perp \oplus \eta$ is decomposed into one-dimensional bundles $\tau_1,\ldots,\tau_\nu$ with respective Stiefel-Whitney classes $1+t_1, \ldots, 1 + t_\nu$. Let the Stiefel-Whitney class of the representation $A_q$ in the cohomology $H^*(B\Sy_q, Z_2)$ be 
$$
e(A_q) = 1 + a_1 + \dots + a_{q-1}.
$$
Now the Euler class equals

\begin{equation}
\label{euler-formula}
e(A_q\otimes (\tau_1\oplus\dots\oplus\tau_\nu) ) = \prod_{i=1}^\nu (t_i^{q-1} + t_i^{q-2}a_1+\dots + a_{q-1})
\end{equation}

in the cohomology $H^*(B\Sy_q\times M, Z_2)$. Then we have to map this class to $H^*(K^q(\mathbb R^\mu)\times M, Z_2)$ or $H^*(Q^q(\mathbb R^\mu)\times M, Z_2)$, by the natural map 
$$
K^q(\mathbb R^\mu)/\Sg_q\to B\Sg_q,\ \text{or}\ Q^q(\mathbb R^\mu)/\Sy_q\to B\Sy_q,
$$ 
find the coefficient at $a_{q-1}^{\mu-1}$, and express it in the Stiefel-Whitney classes of $\eta-\xi$. Of course, the knowledge of the cohomology of the symmetric group modulo $2$ and the relations on these cohomology that describe $H^*(Q^q(\mathbb R^\mu)/\Sy_q)$ should be known.

Passing to the particular case $q=4$ note that $\Sy_4$ is the square group $D_8$, and its cohomology is multiplicatively generated by three elements $a, b, c$ such that 
$$
\dim a = \dim c = 1,\quad \dim b = 2,
$$
and the relation $ac=0$. The Stiefel-Whitney class of $A_4$ is
$$
w(A_4) = (1 + (a+c) + b)(1 + c).
$$
The space $Q^q(\mathbb R^n)$ is a product of three $n-1$-spheres, and in the cohomology of $Q^q(\mathbb R^n)/\Sy_q$ we have relations
$$
c^n=b^n=0.
$$ 
Now (\ref{euler-formula}) has the form
$$
e(A_q\otimes (\tau_1\oplus\dots\oplus\tau_\nu) ) = \prod_{i=1}^\nu (t_i^2 + (a+c)t_i+b)(t_i+c).
$$
to find the leading characteristic class of coincident $4$-tuples in codimension $d$ we have to find the coefficient at $(bc)^{\nu-d-1}$ after applying all the relations. We can also add the artificial relation $a=0$ to simplify the situation, since it does not affect anything. Then some direct calculations give, for example
$$
s_{4,0} = w_1^3+w_1w_2,
$$
$$
s_{4,1} = w_2^3 + w_3^2 + w_1w_2w_3 + w_2w_4.
$$
Further cases can be considered too, it even seems plausible to have some explicit formula with summation for $q=4$.

\section{Local coincident $p$-tuples for prime $p$}
\label{p-fold}

Let us consider local coincident $p$-tuples for odd prime $p$. In this case we consider the Euler class $e(A_p\otimes(\xi^\perp\oplus \eta))$ modulo $p$, and therefore we may consider, instead of $\Sg_p$ its $p$-Sylow subgroup $Z_p$ of cyclic permutations. This group acts on $A_p$ without change of orientation, so we do not have to worry about the twisted cohomology coefficients.

Similar to (\ref{euler-formula}), by the splitting principle we decompose $\xi^\perp\oplus \eta$ of dimension $\nu$ into the sum of $k$ two-dimensional oriented (because we do everything $\mod p$) bundles $\sigma_1\oplus\dots\oplus\sigma_k$, when $\nu=2k$, or into the sum $\sigma_1\oplus\dots\oplus\sigma_k\oplus\tau$ with $\dim\tau=1$, when $\nu=2k+1$. In either case we have
$$
e(A_p\otimes(\xi^\perp\oplus \eta)) = \prod_{i=1}^k e(A_p\otimes \sigma_i) = \prod_{i=1}^k (u^2-e_i^{p-1}),
$$
for even $\nu$, and 
$$
e(A_p\otimes(\xi^\perp\oplus \eta)) = u\prod_{i=1}^k (u^2-e_i^{p-1})
$$
for odd $\nu$. In the last two formulas $u=e(A_p)$, and $e_i=e(\sigma_i)$ are the Euler classes of summands. The formula follows from the formula of Pontryagin classes of a tensor product along with the fact that the only nonzero characteristic classes of $A_p$ in $H^*(BZ_p, Z_p)$ are its Euler class $u$ and its topmost Pontryagin class $u^2$.

Let us define the following characteristic classes by the splitting principle: if the Pontryagin classes are expressed through symmetric functions
$$
p_{4i} = \sigma_i(t_1^2, \ldots, t_k^2)
$$
Then put 
$$
\alpha_{p,i} = \sigma_i(t_1^{p-1}, \ldots, t_k^{p-1}),
$$
in the case $p=3$ these are the Pontryagin classes again, in the general case these are some classes of dimension $2(p-1)i$. Now we can rewrite
$$
e(A_p\otimes(\xi^\perp\oplus \eta)) = \sum_{i=0}^k (-1)^i u^{2(k-i)}\alpha_{p,i}(\xi^\perp\oplus \eta)
$$
for even $\nu$ and 
$$
e(A_p\otimes(\xi^\perp\oplus \eta)) = \sum_{i=0}^k (-1)^i u^{2(k-i)+1}\alpha_{p,i}(\xi^\perp\oplus \eta)
$$
for odd $\nu$.

The images of $u^i$ in $H^{(p-1)i}(K^p(\mathbb R^\mu), Z_p)$ are nonzero if $i \le \mu-1$. Hence we obtain a theorem.

\begin{thm}
\label{p-fold-coinc}
The characteristic classes of coincident $p$-tuples in a fiberwise map $\xi\to\eta$ are 
the classes $\alpha_{p,i}(\eta-\xi)$ with $2i\ge \dim\eta-\dim\xi + 1$.
\end{thm}

\section{Estimating the equivariant category of configuration spaces}
\label{genus-cat}

In the previous sections some cohomology classes of the configuration space $K^q(M)$ were established to be nonzero, thus bounding from below the multiplicity of a map.

Note that the same classes give new lower bounds for the Lyusternik-Schnirelmann category of $K^q(M)/\Sg_q$ and the Krasnosel'skii-Schwarz genus of $K^q(M)$, thus improving the results of~\cite{vass1988,ossa1996,rot2008,kar2009}, where the estimate was made depending on the dimension of $M$ only. 

Let us remind some definitions and lemmas, mainly from~\cite{schw1966}.

\begin{defn}
Let $X$ be a free $G$-space, the \emph{genus} of $X$ is the minimal size of $G$-invariant open cover (i.e. cover by $G$-invariant open subsets) $\{X_1,\ldots,X_n\}$ of $X$ such that every $X_i$ can be $G$-mapped to $G$. Denote the genus of $X$ by $g(X)$.
\end{defn}

It is also well-known that the genus $g(X)$ estimates the Lyusternik-Schnirelmann category $\cat X/G$ from below. We need the following lemma:

\begin{lem}
\label{gen-by-coh}
If $X$ is a paracompact free $G$-space, and for some $G$-module $\alpha$ the natural cohomology map
$$
\pi_X^* : H^n(BG, \alpha)\to H_G^n(X, \alpha)
$$
is nontrivial, then 
$$
\cat X/G \ge g(X)\ge n+1.
$$
\end{lem}

Now we can state a special case of the previous lemma.

\begin{lem}
\label{gen-by-euler}
If the natural image of $e(A_q)^n$ is the cohomology of $K^q(M)$ (or $K_M^q(TM)$) is nontrivial, then 
$$
\cat K^q(M)/\Sg_q \ge g(K^q(M))\ge (q-1)n + 1.
$$
\end{lem}

\begin{proof}
Note that we have an $\Sg_q$-equivariant map 
$$
\exp : K_M^q(TM)\to K^q(M).
$$
If $e(A_q)^n\not=0\in K_M^q(TM)$, then $e(A_q)^n\not=0\in K^q(M)$, and then we apply Lemma~\ref{gen-by-coh}.
\end{proof}

Now we can state some corollaries of Theorems~\ref{2k-fold-leading}, \ref{proj-coinc}, \ref{p-fold-coinc}.

\begin{cor}
Suppose that the dual Stiefel-Whitney class of $M$ with $\dim M=m$ has the form
$$
w(TM^\perp) = 1+\bar w_1+\dots+\bar w_{d+1},
$$
$q$ is a power of two, and $\bar w_{d+1}^{q-1}\not=0$. Then 
$$
g(K^q(M))\ge (m + d)(q-1) + 1.
$$

In particular if $q(d+1) < 2^l-1$ then
$$
g(K^q(RP^{2^l-2-d}))\ge (2^l-3)(q-1) + 1.
$$
\end{cor}

\begin{proof}
Denote the dimension of normal bundle $k=\dim TM^\perp$. Theorem~\ref{2k-fold-leading} claims that the Euler class 
$$
e(A_q\otimes (TM^\perp\oplus\epsilon^{m+d}))\in H^{(q-1)(k + m + d)}(K^q(\mathbb R^{m+k})\times M)
$$ 
is nonzero under the assumptions.

Lemma~\ref{triv-reduction} now claims that the Euler class $e(A_q\otimes \epsilon^{m+d})=e(A_q)^{m+d}$ is also nonzero in $K_M^q(TM)$. Now the result follows from Lemma~\ref{gen-by-euler}.
\end{proof}

The results of Section~\ref{4-fold-coinc} give a similar corollary.

\begin{cor}
Denote $\bar w_i$ the Stiefel-whitney classes of the normal bundle $TM^\perp$. If the class
$$
\bar w_1^3+\bar w_1\bar w_2
$$
is nonzero on $M$ then
$$
g(K^4(M))\ge 3m + 1.
$$
If the class
$$
\bar w_2^3 + \bar w_3^2 + \bar w_1\bar w_2\bar w_3 + \bar w_2\bar w_4
$$
is nonzero on $M$ then
$$
g(K^4(M))\ge 3m + 4.
$$
\end{cor}

And here is the corresponding corollary of Theorem~\ref{p-fold-coinc}.

\begin{cor}
Let $p$ be a prime. Consider the classes $\bar \alpha_{p,i}$ of the normal bundle $TM^\perp$, introduced in Section~\ref{p-fold}. If $\bar \alpha_{p, i}\not=0$ for some $i$, then 
$$
g(K^p(M))\ge (m + 2i - 1)(p-1) + 1.
$$
\end{cor}

\end{document}